\documentclass[a4paper]{amsart}
\usepackage{amscd}
\usepackage{amsmath}
\usepackage{amssymb}
\usepackage{amsthm}
\usepackage{bbm}
\usepackage{stmaryrd}

\usepackage[T1]{fontenc}

\newif\ifpdf
\ifx\pdfoutput\undefined
   \pdffalse        % we are not running PDFLaTeX
\else
   \pdfoutput=1     % we are running PDFLaTeX
   \pdftrue
\fi

\ifpdf
   \usepackage[pdftex]{graphicx}
\else
   \usepackage{graphicx}
\fi

\numberwithin{equation}{section} \swapnumbers

\newtheorem{satz}{Satz}[section]

\newtheorem{theorem}[satz]{Theorem}
\newtheorem{proposition}[satz]{Proposition}
\newtheorem{corollary}[satz]{Corollary}
\newtheorem{lemma}[satz]{Lemma}
\newtheorem{assumption}[satz]{Assumption}

\newtheorem{remark}[satz]{Remark}

\newcommand{\bbr}{\mathbb{R}}
\newcommand{\bbe}{\mathbb{E}}
\newcommand{\bbn}{\mathbb{N}}
\newcommand{\bbp}{\mathbb{P}}

\newcommand{\cald}{\mathcal{D}}

\newcommand{\calf}{\mathcal{F}}

\begin{document}

\hyphenation{Swit-zer-land Ma-ru-nou-chi}

\title[Wong-Zakai approximations with convergence rate]{Wong-Zakai approximations with convergence rate for stochastic partial differential equations}
\author{Toshiyuki Nakayama \and Stefan Tappe}
\address{MUFG Bank, Ltd., Otemachi Financial City Grand Cube 20F 9-2, Otemachi 1-chome, Chiyoda-ku, Tokyo 100-0004, Japan}
\email{tnkym376736@ch.em-net.ne.jp}
\address{Albert Ludwig University of Freiburg, Department of Mathematical Stochastics, Ernst-Zermelo-Stra\ss{}e 1, D-79104 Freiburg, Germany}
\email{stefan.tappe@math.uni-freiburg.de}
\begin{abstract}
The goal of this paper is to prove a convergence rate for Wong-Zakai approximations of semilinear stochastic partial differential equations driven by a finite dimensional Brownian motion. Several examples, including the HJMM equation from mathematical finance, illustrate our result.
\end{abstract}
\maketitle

\section{Introduction}

Consider a semilinear stochastic partial differential equation (SPDE) of the form
\begin{align}\label{SPDE-Wong-Zakai}
\left\{
\begin{array}{rcl}
dX(t) & = & \big( AX(t) + b(X(t)) \big) dt + \sum_{j=1}^r \sigma_j(X(t)) dB^j(t) \medskip
\\ X(0) & = & x_0
\end{array}
\right.
\end{align}
on a separable Hilbert space $(H,\| \cdot \|)$ driven by a finite dimensional Brownian motion $B = (B^1,\ldots,B^r)$ for some positive integer $r \in \bbn$.

A natural method in order to approximate the SPDE (\ref{SPDE-Wong-Zakai}) by a sequence of partial differential equations (PDEs) is to use the so-called Wong-Zakai approximations. More precisely, on a fixed time interval $[0,T]$ we replace the Brownian motions $B^j$ by their polygonal approximations $(B_m^j)_{m \in \mathbb{N}}$ with step size $\frac{T}{m}$. For each $m \in \mathbb{N}$ the Wong-Zakai approximation $\xi_m(\cdot) = \xi_m(\cdot,\omega) : [0,T] \rightarrow H$ is the mild solution to the deterministic PDE
\begin{align}\label{WZ-PDE-intro}
\left\{
\begin{array}{rcl}
\dot{\xi}_m(t) & = & A \xi_m(t) + b(\xi_m(t)) - \frac{1}{2} \sum_{j=1}^r D \sigma_j(\xi_m(t)) \sigma_j(\xi_m(t)) 
\\ &&+ \sum_{j=1}^r \sigma_j(\xi_m(t)) \dot{B}_m^j(t) \medskip
\\ \xi_m(0) & = & x_0
\end{array}
\right.
\end{align}
for each $\omega \in \Omega$. Under appropriate regularity conditions, the Wong-Zakai approximations $(\xi_m)_{m \in \mathbb{N}}$ converge to the solution $X$ of the SPDE (\ref{SPDE-Wong-Zakai}). More precisely, for every $p > 1$ we have the convergence
\begin{align}\label{WZ-intro}
\lim_{m \rightarrow \infty} \mathbb{E} \bigg[ \sup_{t \in [0,T]} \| \xi_m(t) - X(t) \|^{2p} \bigg] = 0,
\end{align}
see \cite[Thm. 2.1]{Nakayama-Support}.

Such a convergence result has first been proven, in the case of finite dimensional SDEs, by Wong and Zakai, see \cite{Wong-Zakai-1, Wong-Zakai-2}. Their approximation result has been generalized into several directions; namely, to the infinite dimensional case, e.g. in \cite{Ac-Terreni, Aida, Brzezniak-Carroll, Brzezniak-Flandoli, Ganguly, Gyongy, Gyongy-Stinga, Hairer, Hausenblas, Nakayama-Support, Tessitore} and \cite{Twardowska-1992}--\cite{Twardowska1-1996b}, with a view to support theorems, e.g. in \cite{Aida, Bally, Gyongy-Nualart-SS, Gyongy-Prohle, Millet-SS-1994, Nakayama-Support} (we also mention the related viability result from \cite{Nakayama}), with a view to the theory of rough paths, e.g. in \cite{Friz-Oberhauser}, for driving processes with jumps, e.g. in \cite{Hausenblas, Proppe}, and with a driving fractional Brownian motion, e.g. in \cite{Tudor-2009}.

However, there are only very few reference dealing with convergence rates for the Wong-Zakai approximations. In \cite{Gyongy} and \cite{Gyongy-Stinga} the authors consider the particular situation where the SPDE (\ref{SPDE-Wong-Zakai}) is a second-order SPDE of parabolic type, and in \cite{Hausenblas} it is assumed that the operator $A$ appearing in (\ref{SPDE-Wong-Zakai}) is the infinitesimal generator of a compact and analytic semigroup.

Our goal in the present paper is to establish a convergence rate for (\ref{WZ-intro}) without imposing restrictions on the generator $A$ appearing in (\ref{SPDE-Wong-Zakai}), that is, $A$ is allowed to be the infinitesimal generator of an arbitrary strongly continuous semigroup.

In order to present our main result, let us briefly outline the assumptions on the drift $b : H \to H$ and the volatilities $\sigma_1,\ldots,\sigma_r : H \to H$; the precise mathematical framework is stated in Section \ref{sec-framework}. First, we assume that these coefficients satisfy standard regularity conditions:

\begin{assumption}\label{ass-H}
We suppose that the following conditions are fulfilled:
\begin{enumerate}
\item The drift $b$ is Lipschitz continuous and bounded.

\item We have $\sigma_j \in C_b^2(H)$ for each $j = 1,\ldots,r$.
\end{enumerate}
\end{assumption}

Here $C_b^2(H)$ denotes the space of all $\sigma \in C^2(H)$ such that $\sigma$, $D \sigma$ and $D^2 \sigma$ are bounded. Then the volatilities $\sigma_1,\ldots,\sigma_r$ are Lipschitz continuous and bounded, and the mapping
\begin{align}\label{def-rho}
\rho : H \to H, \quad \rho(x) := \sum_{j=1}^r D \sigma_j(x) \sigma_j(x)
\end{align}
appearing in the PDE (\ref{WZ-PDE-intro}) is Lipschitz continuous and bounded, too, which ensures existence and uniqueness of mild solutions to the SPDE (\ref{SPDE-Wong-Zakai}) and the PDE (\ref{WZ-PDE-intro}). 

Furthermore, we assume that the conditions stated above are also fulfilled when we consider the coefficients as mappings on the domain $\cald(A)$ of the generator with respect to the graph norm
\begin{align}\label{graph-norm}
\| x \|_{\cald(A)} = \sqrt{\| x \|^2 + \| Ax \|^2}, \quad x \in \cald(A).
\end{align}
More precisely:

\begin{assumption}\label{ass-D}
We suppose that the following conditions are fulfilled:
\begin{enumerate}
\item We have $b(\cald(A)) \subset \cald(A)$ and $\sigma_j(\cald(A)) \subset \cald(A)$ for each $j = 1,\ldots,r$.

\item The drift $b|_{\cald(A)}$ is Lipschitz continuous and bounded with respect to the graph norm $\| \cdot \|_{\cald(A)}$.

\item We have $\sigma_j|_{\cald(A)} \in C_b^2(\cald(A))$ for each $j = 1,\ldots,r$ with respect to the graph norm $\| \cdot \|_{\cald(A)}$.
\end{enumerate}
\end{assumption}

Then our main result reads as follows:

\begin{theorem}\label{thm-main}
Suppose that Assumptions \ref{ass-H} and \ref{ass-D} are fulfilled, and let $T > 0$, $p > 1$ and $x_0 \in \cald(A)$ be arbitrary. Then there is a constant $C > 0$ such that for each $m \in \bbn$ we have
\begin{align*}
\bbe \bigg[ \sup_{t \in [0,T]} \| \xi_m(t) - X(t) \|^{2p} \bigg] \leq \frac{C}{m^{p-1}},
\end{align*}
where $X$ denotes the mild solution to the SPDE (\ref{SPDE-Wong-Zakai}) with $X(0) = x_0$, and the $(\xi_m)_{m \in \bbn}$ denote the mild solutions to the PDEs (\ref{WZ-PDE-intro}) with $\xi_m(0) = x_0$.
\end{theorem}

The proof of Theorem \ref{thm-main} will be a consequence of the following two results:
\begin{enumerate}
\item First, we will prove the stated convergence rate for the Euler-Maruyama approximations; see Theorem \ref{thm-Euler}.

\item Then, we will prove the stated convergence rate for the difference between the Euler-Maruyama approximations and the Wong-Zakai approximations; see Theorem \ref{thm-Euler-WZ}.
\end{enumerate}
For both steps, we will use and extend some results from \cite{Nakayama-Support}. 

The remainder of this text is organized as follows. In Section \ref{sec-framework} we introduce the mathematical framework and present some preliminary results. In Section \ref{sec-Euler} we provide the stated convergence rate for the Euler-Maruyama approximations, and in Section \ref{sec-Euler-WZ} we provide the stated convergence rate for the difference between the Euler-Maruyama approximations and the Wong-Zakai approximations. In the remaining sections we present examples of SPDEs where our main result (Theorem~\ref{thm-main}) applies. These are the HJMM equation from mathematical finance in Section \ref{sec-HJMM}, and two further examples arising from natural sciences in Section \ref{sec-natural}.

\section{General framework and notation}\label{sec-framework}

In this section, we introduce the mathematical framework and present some preliminary results. Let $(\Omega,\calf,(\calf_t)_{t \in \bbr_+},\bbp)$ be a filtered probability space satisfying the usual conditions. Let $B^1,\ldots,B^r$ be independent standard Brownian motions for some positive integer $r \in \bbn$. Let $H$ be a separable Hilbert space and let $(S_t)_{t \geq 0}$ be a $C_0$-semigroup on $H$ with infinitesimal generator $A : \cald(A) \subset H \to H$. Furthermore, let $b : H \to H$ and $\sigma_1,\ldots,\sigma_r : H \to H$ be measurable mappings.

\begin{lemma}\label{lemma-rho}
Suppose that $\sigma_1,\ldots,\sigma_r \in C_b^2(H)$. Then the mapping $\rho : H \to H$ defined in (\ref{def-rho}) is Lipschitz continuous and bounded.
\end{lemma}

\begin{proof}
By assumption, there exists a constant $C > 0$ such that
\begin{align*}
\max \{ \| \sigma_j(x) \|, \| D \sigma_j(x) \|, \| D^2 \sigma_j(x) \| \} \leq C \quad \text{for all $x \in H$ and $j = 1,\ldots,r$.}
\end{align*}
Therefore, for each $x \in H$ we obtain
\begin{align*}
\| \rho(x) \| \leq \sum_{j=1}^r \| D \sigma_j(x) \sigma_j(x) \| \leq \sum_{j=1}^r \| D \sigma_j(x) \| \, \| \sigma_j(x) \| \leq r C^2,
\end{align*}
proving that $\rho$ is bounded. Now, let $x_1,x_2 \in H$ be arbitrary. Then we have
\begin{align*}
&\| \rho(x_1) - \rho(x_2) \| \leq \sum_{j=1}^r \| D \sigma_j(x_1) \sigma_j(x_1) - D \sigma_j(x_2) \sigma_j(x_2) \|
\\ &\leq \sum_{j=1}^r \| D \sigma_j(x_1) \| \, \| \sigma_j(x_1) - \sigma_j(x_2) \| + \sum_{j=1}^r \| \sigma_j(x_2) \| \, \| D \sigma_j(x_1) - D \sigma_j(x_2) \|
\\ &\leq 2rC^2 \| x_1 - x_2 \|,
\end{align*}
showing that $\rho$ is Lipschitz continuous.
\end{proof}

We fix a finite time horizon $T > 0$, and define the quantities
\begin{align*}
\delta_m &:= \frac{T}{m}, \quad m \in \mathbb{N},
\\ [t]_m^- &:= k \delta_m, \quad k \delta_m \leq t < (k+1)\delta_m, \quad k = 0,\ldots,m-1,
\\ [t]_m^+ &:= (k+1) \delta_m, \quad k \delta_m \leq t < (k+1)\delta_m, \quad k = 0,\ldots,m-1,
\end{align*}
and the real-valued processes $(B_m^j(t))_{t \in [0,T]}$ for $m \in \mathbb{N}$ and $j=1\ldots,r$ as
\begin{align*}
B_m^j(t) := B^j([t]_m^-) + \frac{t - [t]_m^-}{\delta_m} (B^j([t]_m^+) - B^j([t]_m^-)), \quad t \in [0,T].
\end{align*}
Note that for all $m \in \bbn$ we have
\begin{align*}
[t]_m^- \leq t < [t]_m^+, \quad t \in [0,T],
\end{align*}
and that for all $m \in \bbn$ and all $j=1\ldots,r$ we have
\begin{align}\label{derivative}
\dot{B}_m^j(t) = \frac{B^j([t]_m^+) - B^j([t]_m^-)}{\delta_m}, \quad t \in [0,T].
\end{align}
For what follows, we suppose that Assumptions \ref{ass-H} and \ref{ass-D} are fulfilled, and consider the SPDE
\begin{align}\label{SPDE}
\left\{
\begin{array}{rcl}
dX(t) & = & \big( A X(t) + \hat{b}(X(t)) \big) dt + \sum_{j=1}^r \sigma_j(X(t)) dB^j(t) \medskip
\\ X(0) & = & x_0,
\end{array}
\right.
\end{align}
where the mapping $\hat{b} : H \to H$ is given by $\hat{b} := b + \frac{\rho}{2}$ with $\rho : H \to H$ being defined in (\ref{def-rho}), and for each $m \in \mathbb{N}$ we consider the Wong-Zakai approximation
$\xi_m(\cdot) = \xi_m(\cdot,\omega) : [0,T] \rightarrow H$ given by the PDE
\begin{align}\label{WZ-approx}
\left\{
\begin{array}{rcl}
\dot{\xi}_m(t) & = & A \xi_m(t) + b(\xi_m(t)) + \sum_{j=1}^r \sigma_j(\xi_m(t)) \dot{B}_m^j(t) \medskip
\\ \xi_m(0) & = & x_0
\end{array}
\right.
\end{align}
for each $\omega \in \Omega$.

\begin{remark}
From now on, we consider the SPDE (\ref{SPDE}) and the PDEs (\ref{WZ-approx}) instead of (\ref{SPDE-Wong-Zakai}) and (\ref{WZ-PDE-intro}). For our purposes, this is more convenient, as then we are directly in the framework of \cite[Sec. 2]{Nakayama-Support}, and it does not mean a restriction by virtue of Lemma \ref{lemma-rho}.
\end{remark}

For each $m \in \bbn$ we define the Euler-Maruyama approximation $Y_m$ inductively as follows. We set $Y_m(0) := x_0$, and, provided that $Y_m$ is defined on the interval $[0,k \delta_m]$ for some $k \in \{ 0,\ldots,m-1 \}$, we set
\begin{equation}\label{Euler-def}
\begin{aligned}
Y_m(t) &:= S_{t - k \delta_m} Y_m(k \delta_m) + \int_{k \delta_m}^t S_{t-s} \hat{b}(Y_m(k \delta_m)) ds
\\ &\quad + \sum_{j=1}^r \int_{k \delta_m}^t S_{t-s} \sigma_j(Y_m(k \delta_m)) dB^j(s), \quad t \in [k \delta_m,(k+1) \delta_m].
\end{aligned}
\end{equation}
Note that for $A = 0$ (that is $S_t = {\rm Id}$ for each $t \geq 0$) the processes $(Y_m)_{m \in \bbn}$ coincide with the well-known Euler-Maruyama approximations with step sizes $\delta_m$ for SDEs.

\begin{remark}
As pointed out in \cite{Bayer}, the naive implementation
\begin{equation}\label{Euler-naive}
\begin{aligned}
Y_m(t) &:= Y_m(k \delta_m) + \int_{k \delta_m}^t \big( A Y_m(k \delta_m) + \hat{b}(Y_m(k \delta_m)) \big) ds
\\ &\quad + \sum_{j=1}^r \int_{k \delta_m}^t \sigma_j(Y_m(k \delta_m)) dB^j(s), \quad t \in [k \delta_m,(k+1) \delta_m].
\end{aligned}
\end{equation}
of the Euler-Maruyama method does not work, because it might immediately lead to some $Y_m(k \delta_m) \notin \cald(A)$. Even in our situation, where we have a well-defined strong solution (see Proposition \ref{prop-ex-sol} below), there is no reason why the discrete approximation (\ref{Euler-naive}) should always stay in $\cald(A)$.
\end{remark}

\begin{lemma}\label{lemma-Euler-acc}
For each $m \in \bbn$ we have
\begin{equation}\label{Euler-formula}
\begin{aligned}
Y_m(t) &= S_t x_0 + \int_0^t S_{t-s} \hat{b}(Y_m([s]_m^-)) ds
\\ &\quad + \sum_{j=1}^r \int_0^t S_{t-s} \sigma_j(Y_m([s]_m^-)) dB^j(s), \quad t \in [0,T].
\end{aligned}
\end{equation}
\end{lemma}

\begin{proof}
We prove identity (\ref{Euler-formula}) inductively on each interval $[0,k \delta_m]$ for $k = 0,\ldots,m$. The identity (\ref{Euler-formula}) holds true for $k = 0$, because $Y_m(0) = x_0$. For the induction step $k \to k+1$ note that for $t = k \delta_m$ identity (\ref{Euler-formula}) yields
\begin{equation}\label{Euler-k}
\begin{aligned}
Y_m(k \delta_m) &= S_{k \delta_m} x_0 + \int_0^{k \delta_m} S_{k \delta_m-s} \hat{b}(Y_m([s]_m^-)) ds
\\ &\quad + \sum_{j=1}^r \int_0^{k \delta_m} S_{k \delta_m-s} \sigma_j(Y_m([s]_m^-)) dB^j(s).
\end{aligned}
\end{equation}
Therefore, by (\ref{Euler-def}) and (\ref{Euler-k}), and noting that
\begin{align*}
[s]_m^- = k \delta_m, \quad s \in [k \delta_m,(k+1) \delta_m), 
\end{align*}
for each $t \in [k \delta_m,(k+1) \delta_m]$ we obtain
\begin{align*}
Y_m(t) &= S_{t - k \delta_m} \bigg( S_{k \delta_m} x_0 + \int_0^{k \delta_m} S_{k \delta_m-s} \hat{b}(Y_m([s]_m^-)) ds
\\ &\quad + \sum_{j=1}^r \int_0^{k \delta_m} S_{k \delta_m-s} \sigma_j(Y_m([s]_m^-)) dB^j(s) \bigg)
\\ &\quad + \int_{k \delta_m}^t S_{t-s} \hat{b}(Y_m([s]_m^-)) ds
+ \sum_{j=1}^r \int_{k \delta_m}^t S_{t-s} \sigma_j(Y_m([s]_m^-)) dB^j(s),
\end{align*}
proving (\ref{Euler-formula}).
\end{proof}

\begin{remark}
With the terminology from \cite{Mild-Ito}, the Euler-Maruyama approximations (\ref{Euler-formula}) are so-called \emph{accelerated exponential Euler approximations}, whereas the so-called \emph{exponential Euler approximations} would be given by
\begin{align*}
Y_m(t) &= S_t x_0 + \int_0^t S_{t-[s]_m^-} \hat{b}(Y_m([s]_m^-)) ds
\\ &\quad + \sum_{j=1}^r \int_0^t S_{t-[s]_m^-} \sigma_j(Y_m([s]_m^-)) dB^j(s), \quad t \in [0,T].
\end{align*}
\end{remark}

The domain $\cald(A)$ equipped with the graph norm (\ref{graph-norm}) is a separable Hilbert space, too, and the restriction $(S_t|_{\cald(A)})_{t \geq 0}$ is a $C_0$-semigroup on $(\cald(A),\| \cdot \|_{\cald(A)})$ with infinitesimal generator $A$ on the domain $\cald(A^2)$. In the upcoming results, the notation $(\cald(A) \, \text{--} ) \int_0^t$ indicates that we consider the respective integral on the state space $(\cald(A),\| \cdot \|_{\cald(A)})$; see, for example, the right-hand sides of (\ref{int-Bochner-3}) and (\ref{int-Bochner-3b}). Otherwise, the integral is considered on the state space $(H,\| \cdot \|)$, as usual; see, for example, the left-hand sides of (\ref{int-Bochner-3}) and (\ref{int-Bochner-3b}).

\begin{lemma}\label{lemma-integrals-H-D-1}
Let $\Phi : [0,T] \to \cald(A)$ be a function such that
\begin{align}\label{int-Bochner-1}
\int_0^T \| \Phi(s) \|_{\cald(A)} ds < \infty.
\end{align}
Then the following statements are true:
\begin{enumerate}
\item We have
\begin{align}\label{int-Bochner-2}
\int_0^T \| \Phi(s) \| ds < \infty.
\end{align}
\item For each $t \in [0,T]$ we have
\begin{align}\label{int-Bochner-3}
\int_0^t \Phi(s) ds = \big( \cald(A) \, \text{--} \big) \int_0^t \Phi(s) ds.
\end{align}
\end{enumerate}
\end{lemma}

\begin{proof}
Relation (\ref{int-Bochner-2}) is an immediate consequence of (\ref{int-Bochner-1}). There is a sequence $(\Phi^n)_{n \in \bbn}$ of simple functions $\Phi^n : [0,T] \to \cald(A)$ such that $\| \Phi - \Phi^n \|_{\cald(A)} \leq \| \Phi \|_{\cald(A)}$ for each $n \in \bbn$, and we have $\| \Phi - \Phi^n \|_{\cald(A)} \to 0$ for $n \to \infty$. Let $t \in [0,T]$ be arbitrary. By Lebesgue's dominated convergence theorem we obtain
\begin{align*}
\bigg\| \big( \cald(A) \, \text{--} \big) \int_0^t \Phi(s) ds - \big( \cald(A) \, \text{--} \big) \int_0^t \Phi^n(s) ds \bigg\|_{\cald(A)} \to 0 \quad \text{as $n \to \infty$.}
\end{align*}
We have $\| \Phi - \Phi^n \| \leq \| \Phi \|_{\cald(A)}$ for each $n \in \bbn$, and $\| \Phi - \Phi^n \| \to 0$ for $n \to \infty$. Therefore, by Lebesgue's dominated convergence theorem we also have
\begin{align*}
\bigg\| \int_0^t \Phi(s) ds - \int_0^t \Phi^n(s) ds \bigg\| \to 0 \quad \text{as $n \to \infty$.}
\end{align*}
Noting that
\begin{align*}
\int_0^t \Phi^n(s) ds = \big( \cald(A) \, \text{--} \big) \int_0^t \Phi^n(s) ds \quad \text{for each $n \in \bbn$,}
\end{align*}
we arrive at (\ref{int-Bochner-3}).
\end{proof}

\begin{lemma}\label{lemma-integrals-H-D-2}
Let $\Psi$ be a $\cald(A)$-valued predictable process such that $\bbp$-almost surely
\begin{align}\label{int-Bochner-1b}
\int_0^T \| \Psi(s) \|_{\cald(A)}^2 ds < \infty.
\end{align}
Then the following statements are true:
\begin{enumerate}
\item We have
\begin{align}\label{int-Bochner-2b}
\int_0^T \| \Psi(s) \|^2 ds < \infty.
\end{align}
\item For each $t \in [0,T]$ and each $j = 1,\ldots,r$ we have
\begin{align}\label{int-Bochner-3b}
\int_0^t \Psi(s) dB^j(s) = \big( \cald(A) \, \text{--} \big) \int_0^t \Psi(s) dB^j(s).
\end{align}
\end{enumerate}
\end{lemma}

\begin{proof}
The proof is similar to that of Lemma \ref{lemma-integrals-H-D-1}, and therefore omitted.
\end{proof}

The following three results show that for each starting point $x_0 \in \cald(A)$ the mild solution $X$ to the SPDE (\ref{SPDE}) with $X(0) = x_0$, the  Wong-Zakai approximations $(\xi_m)_{m \in \bbn}$ given by the PDEs (\ref{WZ-approx}) with $\xi_m(0) = x_0$ and the Euler-Maruyama approximations $(Y_m)_{m \in \bbn}$ given by $Y_m(0) = x_0$ and (\ref{Euler-def}) take their values in $\cald(A)$.

\begin{proposition}\label{prop-ex-sol}
For each $x_0 \in \cald(A)$ there exists a unique mild solution $X$ to the SPDE (\ref{SPDE}) on the state space $(\cald(A),\| \cdot \|_{\cald(A)})$, and it is a strong solution to the SPDE (\ref{SPDE}) on the state space $H$.
\end{proposition}

\begin{proof}
By a standard result (see, for example, \cite[Thm. 7.2]{Da_Prato}), there is a a unique mild solution $X$ to the SPDE (\ref{SPDE}) on the state space $(\cald(A),\| \cdot \|_{\cald(A)})$; that is, a continuous adapted process such that $\bbp$-almost surely
\begin{align*}
X(t) &= S_t x_0 + \big( \cald(A) \, \text{--} \big) \int_0^t S_{t-s} \hat{b}(X(s)) ds
\\ &\quad + \sum_{j=1}^r \big( \cald(A) \, \text{--} \big) \int_0^t S_{t-s} \sigma_j(X(s)) d B^j(s), \quad t \in \bbr_+.
\end{align*}
This implies that $X$ is also continuous in $H$, and by Lemmas \ref{lemma-integrals-H-D-1} and \ref{lemma-integrals-H-D-2} we obtain $\bbp$-almost surely
\begin{align*}
X(t) &= S_t x_0 + \int_0^t S_{t-s} \hat{b}(X(s)) ds
\\ &\quad + \sum_{j=1}^r \int_0^t S_{t-s} \sigma_j(X(s)) d B^j(s), \quad t \in \bbr_+,
\end{align*}
showing that $X$ is also a mild solution to the SPDE (\ref{SPDE}) on the state space $H$. Furthermore, since $X$ is continuous in $(\cald(A),\| \cdot \|_{\cald(A)})$ we have $\bbp$-almost surely
\begin{align*}
\int_0^t \| A X(s) \| ds \leq \int_0^t \| X(s) \|_{\cald(A)} ds < \infty \quad \text{for each $t \in \bbr_+$.}
\end{align*}
Therefore $X$ is also a strong solution to the SPDE (\ref{SPDE}) on the state space $H$.
\end{proof}

\begin{proposition}\label{prop-ex-sol-WZ}
For each $x_0 \in \cald(A)$ and each $m \in \bbn$ there exists a unique mild solution $\xi_m$ to the PDE (\ref{WZ-approx}) on the state space $(\cald(A),\| \cdot \|_{\cald(A)})$, and it is a strong solution to the PDE (\ref{WZ-approx}) on the state space $H$. Moreover, for each $k \in \{ 0,\ldots,m-1 \}$ we have
\begin{align*}
\xi_m(t) &= S_{t - k \delta_m} \xi_m(k \delta_m) + \int_{k \delta_m}^t S_{t-s} b(\xi_m(s)) ds
\\ &\quad + \sum_{j=1}^r \int_{k \delta_m}^t S_{t-s} \sigma_j(\xi_m(s)) \dot{B}_m^j(s) ds, \quad t \in [k \delta_m,(k+1) \delta_m].
\end{align*}
\end{proposition}

\begin{proof}
The proof is similar to that of Proposition \ref{prop-ex-sol} and Lemma \ref{lemma-Euler-acc}, and therefore omitted.
\end{proof}

\begin{proposition}
For each $x_0 \in \cald(A)$ and each $m \in \bbn$ we have $Y_m \in \cald(A)$.
\end{proposition}

\begin{proof}
Noting that $Y_m(0) = x_0$ and (\ref{Euler-def}), this is an immediate consequence of Lemmas \ref{lemma-integrals-H-D-1} and \ref{lemma-integrals-H-D-2}.
\end{proof}

\begin{lemma}\label{lemma-convolution}
Let $\Psi$ be an $H$-valued predictable process, and let $p > 1$ be such that
\begin{align*}
\bbe \bigg[ \int_0^T \| \Psi(s) \|^{2p} ds \bigg] < \infty.
\end{align*}
Then there is a constant $C > 0$ such that for each $j=1,\ldots,r$ we have
\begin{align*}
\bbe \Bigg[ \sup_{t \in [0,T]} \bigg\| \int_0^T S_{t-s} \Psi(s) dB^j(s) \bigg\|^{2p} \Bigg] \leq C \, \bbe \bigg[ \int_0^T \| \Psi(s) \|^{2p} ds \bigg].
\end{align*}
\end{lemma}

\begin{proof}
This follows, for example, from \cite[Lemma 3.3]{Atma-book}.
\end{proof}

\begin{lemma}\label{lemma-est-semigroup}
There is a constant $C > 0$ such that for all $t_1,t_2 \in [0,T]$ with $t_1 \leq t_2$ and all $x \in \mathcal{D}(A)$ we have
\begin{align}\label{est-semigroup}
\| S_{t_2} x - S_{t_1} x \| \leq C \| x \|_{\mathcal{D}(A)} |t_2 - t_1|.
\end{align}
\end{lemma}

\begin{proof}
According to \cite[Thm. 2.2]{Pazy} there are constants $M \geq 1$ and $\omega \in \mathbb{R}$ such that
\begin{align*}
\| S_t \| \leq M e^{\omega t} \quad \text{for all $t \geq 0$.}
\end{align*}
Therefore, by \cite[Thm. 2.4]{Pazy} we obtain
\begin{align*}
\| S_{t_2} x - S_{t_1} x \| &= \bigg\| \int_{t_1}^{t_2} S_s Ax ds \bigg\| \leq \int_{t_1}^{t_2} \| S_s Ax \| ds 
\\ &\leq \int_{t_1}^{t_2} \| S_s \| \, \| Ax \| ds  \leq Me^{\omega T} \| x \|_{\mathcal{D}(A)} |t_2 - t_1|,
\end{align*}
proving (\ref{est-semigroup}) with $C = M e^{\omega T}$.
\end{proof}

\section{Convergence rate for the Euler-Maruyama approximations}\label{sec-Euler}

In this section, we prove the stated convergence rate for the Euler-Maruyama approximations. The general mathematical framework is that of Section \ref{sec-framework}.

\begin{theorem}\label{thm-Euler}
Suppose that Assumptions \ref{ass-H} and \ref{ass-D} are fulfilled, and let $T > 0$, $p > 1$ and $x_0 \in \cald(A)$ be arbitrary. Then there is a constant $C > 0$ such that for each $m \in \bbn$ we have
\begin{align*}
\bbe \bigg[ \sup_{t \in [0,T]} \| Y_m(t) - X(t) \|^{2p} \bigg] \leq \frac{C}{m^{p-1}},
\end{align*}
where $X$ denotes the mild solution to the SPDE (\ref{SPDE}) with $X(0) = x_0$, and the $(Y_m)_{m \in \bbn}$ denote Euler-Maruyama approximations given by $Y_m(0) = x_0$ and (\ref{Euler-def}).
\end{theorem}

We will provide the proof of Theorem \ref{thm-Euler} at the end of this section. For each $m \in \bbn$ we introduce the processes $\bar{Y}_m$ and $\bar{X}_m$ as
\begin{align*}
\bar{Y}_m(t) &:= S_{t - [t]_m^-} Y_m([t]_m^-), \quad t \in \bbr_+,
\\ \bar{X}_m(t) &:= S_{t - [t]_m^-} X([t]_m^-), \quad t \in \bbr_+.
\end{align*}
Then, for each $m \in \bbn$ we have
\begin{align}\label{decomp-Euler}
Y_m - X = (Y_m - \bar{Y}_m) + (\bar{Y}_m - \bar{X}_m) + (\bar{X}_m - X).
\end{align}
In the upcoming proofs, we will denote by $C$ a suitable positive constant, possibly different from line to line, but only depending on $T$, $p$, $x_0$ and the parameters $(A,b,\sigma)$ of the SPDE (\ref{SPDE}).

\begin{proposition}\label{prop-Euler-1}
There is a constant $C > 0$ such that for each $m \in \bbn$ we have
\begin{align*}
\bbe \bigg[ \sup_{t \in [0,T]} \| Y_m(t) - \bar{Y}_m(t) \|^{2p} \bigg] \leq C \delta_m^{p-1}.
\end{align*}
\end{proposition}

\begin{proof}
We have
\begin{align*}
Y_m(t) - \bar{Y}_m(t) = \int_{[t]_m^-}^t S_{t-s} \hat{b}(Y_m([s]_m^-)) ds + \sum_{j=1}^r \int_{[t]_m^-}^t S_{t-s} \sigma_j(Y_m([s]_m^-)) dB^j(s).
\end{align*}
Since
\begin{align*}
\bigg\| \int_{[t]_m^-}^t S_{t-s} \hat{b}(Y_m([s]_m^-)) ds \bigg\| \leq C \delta_m,
\end{align*}
applying \cite[Lemma 2.4]{Nakayama-Support} completes the proof.
\end{proof}

\begin{proposition}\label{prop-Euler-2}
There is a constant $C > 0$ such that for each $m \in \bbn$ we have
\begin{align*}
\bbe \bigg[ \sup_{t \in [0,T]} \| \bar{X}_m(t) - X(t) \|^{2p} \bigg] \leq C \delta_m^{p-1}.
\end{align*}
\end{proposition}

\begin{proof}
See \cite[Lemma 2.5]{Nakayama-Support}.
\end{proof}

\begin{lemma}\label{lemma-Euler-event-1}
There is a constant $C > 0$ such that for each $m \in \bbn$ we have
\begin{align*}
\bbe \bigg[ \sup_{t \in [0,T]} \| \bar{Y}_m(t) - Y_m([t]_m^-) \|^{2p} \bigg] \leq C \delta_m^{2p}.
\end{align*}
\end{lemma}

\begin{proof}
Note that
\begin{align*}
\bar{Y}_m(t) - Y_m([t]_m^-) = (S_{t - [t]_m^-} - {\rm Id}) Y_m([t]_m^-).
\end{align*}
Therefore, by Lemma \ref{lemma-est-semigroup} we obtain
\begin{align*}
\bbe \bigg[ \sup_{t \in [0,T]} \| \bar{Y}_m(t) - Y_m([t]_m^-) \|^{2p} \bigg] \leq C \delta_m^{2p} \bbe \bigg[ \sup_{t \in [0,T]} \| Y_m([t]_m^-) \|_{\cald(A)}^{2p} \bigg],
\end{align*}
which, by virtue of Lemma \ref{lemma-convolution} -- applied with the separable Hilbert space $(\cald(A),\| \cdot \|_{\cald(A)})$ -- and Assumption \ref{ass-D} completes the proof.
\end{proof}

\begin{proposition}\label{prop-Euler-3}
There is a constant $C > 0$ such that for each $m \in \bbn$ and each $v \in [0,T]$ we have
\begin{align*}
\bbe \bigg[ \sup_{t \in [0,v]} \| \bar{Y}_m(t) - \bar{X}_m(t) \|^{2p} \bigg] \leq C \bigg( \int_0^v \bbe \bigg[ \sup_{t \in [0,u]} \| Y_m(t) - X(t) \|^{2p} \bigg] du + \delta_m^{p-1} \bigg).
\end{align*}
\end{proposition}

\begin{proof}
Note that
\begin{align*}
&\bar{Y}_m(t) - \bar{X}_m(t) = \int_0^{[t]_m^-} S_{t-s} \big( \hat{b}(Y([s]_m^-)) - \hat{b}(X(s)) \big) ds
\\ &\quad + \sum_{j=1}^r \int_0^{[t]_m^-} S_{t-s} \big( \sigma_j (Y([s]_m^-)) - \sigma_j(X(s)) \big) dB^j(s), \quad t \in [0,T].
\end{align*}
We have
\begin{align*}
&\bbe \Bigg[ \sup_{t \in [0,v]} \bigg\| \int_0^{[t]_m^-} S_{t-s} \big( \hat{b}(Y([s]_m^-)) - \hat{b}(X(s)) \big) ds \bigg\|^{2p} \Bigg]
\\ &\leq C \int_0^v \bbe [ \| Y([s]_m^-) - X(s) \|^{2p} ] ds.
\end{align*}
Furthermore, by Lemma \ref{lemma-convolution} we have
\begin{align*}
&\bbe \Bigg[ \sup_{t \in [0,v]} \bigg\| \int_0^{[t]_m^-} S_{t-s} \big( \sigma_j(Y([s]_m^-)) - \sigma_j(X(s)) \big) dB^j(s) \bigg\|^{2p} \Bigg]
\\ &\leq C \int_0^v \bbe [ \| Y([s]_m^-) - X(s) \|^{2p} ] ds.
\end{align*}
Therefore, we obtain
\begin{align*}
\bbe \bigg[ \sup_{t \in [0,v]} \| \bar{Y}_m(t) - \bar{X}_m(t) \|^{2p} \bigg] &\leq C \big( I_1(v) + I_2(v) + I_3(v) \big),
\end{align*}
where we have set
\begin{align*}
I_1(v) &:= \int_0^v \bbe \bigg[ \sup_{t \in [0,u]} \| Y([t]_m^-) - \bar{Y}_m(t) \|^{2p} \bigg] du,
\\ I_2(v) &:= \int_0^v \bbe \bigg[ \sup_{t \in [0,u]} \| \bar{Y}_m(t) - Y_m(t) \|^{2p} \bigg] du,
\\ I_3(v) &:= \int_0^v \bbe \bigg[ \sup_{t \in [0,u]} \| Y_m(t) - X(t) \|^{2p} \bigg] du.
\end{align*}
Therefore, applying Proposition \ref{prop-Euler-1} and Lemma \ref{lemma-Euler-event-1} completes the proof.
\end{proof}

Now, the proof of Theorem \ref{thm-Euler} is an immediate consequence of the decomposition (\ref{decomp-Euler}), Propositions \ref{prop-Euler-1}, \ref{prop-Euler-2}, \ref{prop-Euler-3} and Gronwall's inequality.

\section{Distance between the Euler-Maruyama approximations and the Wong-Zakai approximations}\label{sec-Euler-WZ}

In this section, we prove the stated convergence rate for the difference between the Euler-Maruyama approximations and the Wong-Zakai approximations. The general mathematical framework is that of Section \ref{sec-framework}.

\begin{theorem}\label{thm-Euler-WZ}
Suppose that Assumptions \ref{ass-H} and \ref{ass-D} are fulfilled, and let $T > 0$, $p > 1$ and $x_0 \in \cald(A)$ be arbitrary. Then there is a constant $C > 0$ such that for each $m \in \bbn$ we have
\begin{align*}
\bbe \bigg[ \sup_{t \in [0,T]} \| \xi_m(t) - Y_m(t) \|^{2p} \bigg] \leq \frac{C}{m^{p-1}},
\end{align*}
where the $(\xi_m)_{m \in \bbn}$ denote the mild solutions to the PDEs (\ref{WZ-approx}) with $\xi_m(0) = x_0$, and the $(Y_m)_{m \in \bbn}$ denote the Euler-Maruyama approximations given by $Y_m(0) = x_0$ and (\ref{Euler-def}).
\end{theorem}

We will provide the proof of Theorem \ref{thm-Euler-WZ} at the end of this section. For each $m \in \bbn$ we introduce the processes $\bar{\xi}_m$ and $\bar{Y}_m$ as
\begin{align*}
\bar{\xi}_m(t) &:= S_{t - [t]_m^-} \xi_m([t]_m^-), \quad t \in \bbr_+,
\\ \bar{Y}_m(t) &:= S_{t - [t]_m^-} Y_m([t]_m^-), \quad t \in \bbr_+.
\end{align*}
Then, for each $m \in \bbn$ we have
\begin{align}\label{decomp-Euler-WZ}
\xi_m - Y_m = (\xi_m - \bar{\xi}_m) + (\bar{\xi}_m - \bar{Y}_m) + (\bar{Y}_m - Y_m).
\end{align}

\begin{proposition}\label{prop-Euler-WZ-pre}
There is a constant $C > 0$ such that for each $m \in \bbn$ we have
\begin{align*}
\bbe \bigg[ \sup_{t \in [0,T]} \| \xi_m - \bar{\xi}_m \|^{2p} \bigg] \leq C \delta_m^{p-1}.
\end{align*}
\end{proposition}

\begin{proof}
See \cite[Lemma 2.2]{Nakayama-Support}.
\end{proof}

We have the identity
\begin{align*}
\bar{\xi}_m(t) - \bar{Y}_m(t) = S_{t - [t]_m^-} \big( \xi_m([t]_m^-) - Y_m([t]_m^-) \big),
\end{align*}
and hence

\begin{equation}\label{xi-X}
\begin{aligned}
\bar{\xi}_m(t) - \bar{Y}_m(t) &= S_{t - [t]_m^-} \bigg( \int_0^{[t]_m^-} S_{t-s} b(\xi_m(s))ds 
\\ &\quad + \sum_{j=1}^r \int_0^{[t]_m^-} S_{t-s} \sigma_j(\xi_m(s)) \dot{B}_m^j(s)ds
- \int_0^{[t]_m^-} S_{t-s} b(Y_m([s]_m^-))ds 
\\ &\quad - \frac{1}{2} \sum_{j=1}^r \int_0^{[t]_m^-} S_{t-s} D \sigma_j(Y_m([s]_m^-)) \sigma_j(Y_m([s]_m^-)) ds
\\ &\quad - \sum_{j=1}^r \int_0^{[t]_m^-} S_{t-s} \sigma_j(Y_m([s]_m^-))d B^j(s) \bigg).
\end{aligned}
\end{equation}
Fix an arbitrary $j \in \{ 1,\ldots,r \}$. By equation (2.6) in \cite{Nakayama-Support}, for each $m \in \mathbb{N}$ we have
\begin{align}\label{sigma-expression}
\sigma_j(\xi_m(t)) = \gamma_{m,1}^j(t) + \sum_{l=1}^r \gamma_{m,2}^{j,l}(t) + \gamma_{m,3}^j(t), \quad t \in [0,T],
\end{align}
where the quantity $\gamma_{m,1}^j(t)$ is given by
\begin{align*}
\gamma_{m,1}^j(t) := \sigma_j(\bar{\xi}_m(t)),
\end{align*}
the quantities $\gamma_{m,2}^{j,l}(t)$ for $l = 1,\ldots,r$ are given by
\begin{align*}
\gamma_{m,2}^{j,l}(t) := D \sigma_j(\bar{\xi}_m(t)) \int_{[t]_m^-}^t S_{t-u} \sigma_l(\bar{\xi}_m(u)) \dot{B}_m^l(u) du,
\end{align*}
and the quantity $\gamma_{m,3}^j(t)$ is given by
\begin{align*}
\gamma_{m,3}^j(t) &:= D \sigma_j(\bar{\xi}_m(t)) \int_{[t]_m^-}^t S_{t-u} b(\xi_m(u)) du
\\ &\quad + \sum_{l=1}^r D \sigma_j(\bar{\xi}_m(t)) \int_{[t]_m^-}^t S_{t-u} \bigg( \int_0^1 D \sigma_l(\bar{\xi}_m(u))
\\ &\quad + v ( \xi_m(u) - \bar{\xi}_m(u) ) ( \xi_m(u) - \bar{\xi}_m(u) ) dv \bigg) \dot{B}_m^l(u) du
\\ &\quad + \int_0^1 \bigg( \int_0^{v_1} D^2 \sigma_j(\bar{\xi}_m(t) + v_2 (\xi_m(t) - \bar{\xi}_m(t)))
\\ &\qquad ( \xi_m(t) - \bar{\xi}_m(t), \xi_m(t) - \bar{\xi}_m(t)) ) dv_2 \bigg) dv_1.
\end{align*}
We introduce the quantity $\zeta_{m,1}^j(t)$ as
\begin{align*}
\zeta_{m,1}^j(t) := \int_0^{[t]_m^-} S_{t-s} \gamma_{m,3}^j(s) \dot{B}_m^j(s)ds,
\end{align*}
the quantity $\zeta_{m,2}^j(t)$ as
\begin{align*}
\zeta_{m,2}^j(t) &:= \int_0^{[t]_m^-} S_{t-s} \gamma_{m,1}^j(s) \dot{B}_m^j(s)ds - \int_0^{[t]_m^-} S_{t-s} \sigma_j(Y_m([s]_m^-))dB^j(s),
\end{align*}
the quantities $\zeta_{m,3}^{j,l}(t)$ for $l = 1,\ldots,r$ as
\begin{align*}
\zeta_{m,3}^{j,j}(t) &:= \int_0^{[t]_m^-} S_{t-s} \gamma_{m,2}^{j,j}(s) \dot{B}_m^j(s)ds
\\ &\quad - \frac{1}{\delta_m} \int_0^{[t]_m^-} S_{t-s} D\sigma_j(\bar{\xi}_m(s)) \bigg( \int_{[s]_m^-}^s S_{s-u} \sigma_j(\bar{\xi}_m(u))du \bigg) ds
\end{align*}
and
\begin{align*}
\zeta_{m,3}^{j,l}(t) &:= \int_0^{[t]_m^-} S_{t-s} \gamma_{m,2}^{j,l}(s) \dot{B}_m^j(s)ds \quad \text{for $l \neq j$,}
\end{align*}
the quantity $\zeta_{m,4}^j(t)$ as
\begin{align*}
\zeta_{m,4}^j(t) &:= \frac{1}{\delta_m} \int_0^{[t]_m^-} S_{t-s} D\sigma_j(\bar{\xi}_m(s)) \bigg( \int_{[s]_m^-}^s S_{s-u} \sigma_j(\bar{\xi}_m(u))du \bigg) ds
\\ &\quad - \frac{1}{2} \int_0^{[t]_m^-} S_{t-s} D \sigma_j(Y_m([s]_m^-)) \sigma_j(Y_m([s]_m^-)) ds,
\end{align*}
and the quantity $\zeta_{m,5}(t)$ as
\begin{align*}
\zeta_{m,5}(t) := \int_0^{[t]_m^-} S_{t-s} \big( b(\xi_m(s)) - b(Y_m([s]_m^-)) \big) ds.
\end{align*}
Then we obtain
\begin{equation}\label{xi-part-3}
\begin{aligned}
&\bar{\xi}_m(t) - \bar{Y}_m(t) = S_{t - [t]_m^-}
\\ &\bigg( \sum_{j=1}^r \zeta_{m,1}^j(t) + \sum_{j=1}^r \zeta_{m,2}^j(t) + \sum_{j=1}^r \sum_{l=1}^r \zeta_{m,3}^{j,l}(t) + \sum_{j=1}^r \zeta_{m,4}^j(t) + \zeta_{m,5}(t) \bigg).
\end{aligned}
\end{equation}

\begin{proposition}\label{prop-WZ-1}
There is a constant $C > 0$ such that for each $m \in \bbn$ we have
\begin{align*}
\mathbb{E} \bigg[ \sup_{t \in [0,T]} \| \zeta_{m,1}^j(t) \|^{2p} \bigg] \leq C \delta_m^p.
\end{align*}
\end{proposition}

\begin{proof}
This follows from \cite[Lemma 2.6]{Nakayama-Support}.
\end{proof}

\begin{proposition}\label{prop-WZ-2}
There is a constant $C > 0$ such that for each $m \in \bbn$ and each
$v \in [0,T]$ we have
\begin{align*}
\mathbb{E} \bigg[ \sup_{t \in [0,v]} \| \zeta_{m,2}^j(t) \|^{2p} \bigg] \leq C \bigg( \int_0^v \bbe \bigg[ \sup_{t \in [0,u]} \| \xi_m(t) - Y_m(t) \|^{2p} \bigg] du + \delta_m^{2p} \bigg).
\end{align*}
\end{proposition}

\begin{proof}
An analogous calculation as in the proof of \cite[Lemma 2.9]{Nakayama-Support} shows that
\begin{align*}
\bbe \bigg[ \sup_{t \in [0,v]} \| \zeta_{m,2}^j(t) \|^{2p} \bigg] \leq C \bbe \bigg[ \int_0^{[v]_m^-} \| \tilde{Y}_m^j(u) \|^{2p} du \bigg], \quad v \in [0,T],
\end{align*}
where
\begin{align*}
\tilde{Y}_m^j(u) = \frac{1}{\delta_m} \int_{[u]_m^-}^{[u]_m^+} S_{[u]_m^+ - \tilde{u}} \sigma_j(\bar{\xi}_m(\tilde{u})) d \tilde{u} - S_{[u]_m^+ - u} \sigma_j(Y_m([u]_m^-)).
\end{align*}
Now, we have
\begin{align*}
\tilde{Y}_m^j(u) &= \frac{1}{\delta_m} \int_{[u]_m^-}^{[u]_m^+} S_{[u]_m^+ - \tilde{u}} \sigma_j(S_{\tilde{u} - [u]_m^-} \xi_m([u]_m^-) ) d \tilde{u}
\\ &\quad - \frac{1}{\delta_m} \int_{[u]_m^-}^{[u]_m^+} S_{[u]_m^+ - u} \sigma_j(Y_m([u]_m^-)) d \tilde{u},
\end{align*}
and hence, we obtain
\begin{align*}
\tilde{Y}_m^j(u) &= \frac{1}{\delta_m} \int_{[u]_m^-}^{[u]_m^+} S_{[u]_m^+ - \tilde{u}} \big( \sigma_j(S_{\tilde{u} - [u]_m^-} \xi_m([u]_m^-) ) - \sigma_j(S_{\tilde{u} - [u]_m^-} Y_m([u]_m^-) ) \big) d \tilde{u}
\\ &\quad + \frac{1}{\delta_m} \int_{[u]_m^-}^{[u]_m^+} S_{[u]_m^+ - \tilde{u}} \big( \sigma_j(S_{\tilde{u} - [u]_m^-} Y_m([u]_m^-) ) - \sigma_j(Y_m([u]_m^-)) \big) d \tilde{u}
\\ &\quad + \frac{1}{\delta_m} \int_{[u]_m^-}^{[u]_m^+} \big( S_{[u]_m^+ - \tilde{u}} - S_{[u]_m^+ - u} \big) \sigma_j(Y_m([u]_m^-)) d \tilde{u}.
\end{align*}
Therefore, and by Lemma \ref{lemma-est-semigroup} and Assumption \ref{ass-D} we obtain
\begin{align*}
\| \tilde{Y}_m^j(u) \| &\leq C \| \xi_m([u]_m^-) - Y_m([u]_m^-) \|
\\ &\quad + C \sup_{\tilde{u} \in [0,\delta_m]} \| ( S_{\tilde{u}} - {\rm Id} ) Y_m([u]_m^-) \| 
\\ &\quad + C \sup_{\genfrac{}{}{0pt}{}{u_1,u_2 \in [0,T]}{|u_1 - u_2| \leq \delta_m}} \| (S_{u_1} - S_{u_2}) \sigma_j(Y_m([u]_m^-)) \|
\\ &\leq C \| \xi_m([u]_m^-) - Y_m([u]_m^-) \| + C \delta_m \big( \| Y_m([u]_m^-) \|_{\cald(A)} + 1 \big).
\end{align*}
By Lemma \ref{lemma-convolution} -- applied with the separable Hilbert space $(\cald(A),\| \cdot \|_{\cald(A)})$ -- and Assumption \ref{ass-D} we obtain
\begin{align*}
&\bbe \bigg[ \sup_{t \in [0,v]} \| \zeta_{m,2}^j(t) \|^{2p} \bigg] \leq C \int_0^{[v]_m^-} \bbe \left[ \| \tilde{Y}_m^j(u) \|^{2p} \right] du
\\ &\leq  C \int_0^v \bigg( \bbe [ \| \xi_m([u]_m^-) - Y_m([u]_m^-) \|^{2p} ] + \delta_m^{2p} \Big( \bbe \left[ \| Y_m([u]_m^-) \|_{\cald(A)}^{2p} \right] + 1 \Big) \bigg) du
\\ &\leq C \bigg( \int_0^v \bbe \bigg[ \sup_{t \in [0,u]} \| \xi_m(t) - Y_m(t) \|^{2p} \bigg] du + \delta_m^{2p} \bigg),
\end{align*}
finishing the proof.
\end{proof}

\begin{lemma}\label{lemma-chi-square}
Let $X \sim {\rm N}(0,\sigma^2)$ be a normally distributed random variable with variance $\sigma^2 > 0$. Then, for each positive real number $q > 0$ we have
\begin{align*}
\bbe \left[ |X|^{2q} \right] = 2^q \frac{\Gamma(q + \frac{1}{2})}{\Gamma(\frac{1}{2})} \sigma^{2q}.
\end{align*}
\end{lemma}

\begin{proof}
We have $X = \sigma Y$ with a random variable $Y \sim {\rm N}(0,1)$. Therefore, we have $Y^2 \sim \chi^2$, and by \cite[Sec. 7.8.1]{Wilks} we obtain
\begin{align*}
\bbe \left[ |X|^{2q} \right] = \bbe \left[ |\sigma Y|^{2q} \right] = \sigma^{2q} \bbe \left[ (Y^2)^q \right] = 2^q \frac{\Gamma(q+\frac{1}{2})}{\Gamma(\frac{1}{2})} \sigma^{2q},
\end{align*}
completing the proof.
\end{proof}

\begin{corollary}\label{cor-chi-square}
Let $q > 0$ be a positive real number. Then there is a constant $C > 0$ such that for each $m \in \bbn$ we have
\begin{align*}
\bbe \Big[ |B^j((k+1)\delta_m) - B^j(k \delta_m)|^{2q} \Big] \leq C \delta_m^q.
\end{align*}
\end{corollary}

\begin{proof}
Since $B^j((k+1)\delta_m) - B^j(k \delta_m) \sim {\rm N}(0,\delta_m)$ for each $k=0,\ldots,m-1$, this is an immediate consequence of Lemma \ref{lemma-chi-square}.
\end{proof}

The proof of the following auxiliary result is similar to that of Lemma 2.2 in \cite{Nakayama-Support}.

\begin{lemma}\label{lemma-E-B-dot}
Let $q > 0$ be a positive real number. Then there is a constant $C > 0$ such that for each $m \in \bbn$ we have
\begin{align*}
\mathbb{E} \bigg[ \sup_{t \in [0,T]} | \dot{B}_m^j(t) |^{2q} \bigg] \leq C \delta_m^{-(q+1)}.
\end{align*}
\end{lemma}

\begin{proof}
Taking into account (\ref{derivative}), by Corollary \ref{cor-chi-square} we obtain
\begin{align*}
&\mathbb{E} \bigg[ \sup_{t \in [0,T]} | \dot{B}_m^j(t) |^{2q} \bigg] = \frac{1}{\delta_m^{2q}} \mathbb{E} \bigg[ \sup_{t \in [0,T]} | B^j([t]_m^+) - B^j([t]_m^-) |^{2q} \bigg]
\\ &\leq \frac{1}{\delta_m^{2q}} \sum_{k=0}^{m-1} \bbe \Big[ |B^j((k+1)\delta_m) - B^j(k \delta_m)|^{2q} \Big] \leq \frac{C}{\delta_m^{2q}} m \delta_m^q = C \delta_m^{-(q+1)},
\end{align*}
completing the proof.
\end{proof}

\begin{corollary}\label{cor-E-B-dot}
Let $q > 0$ be a positive real number. Then there is a constant $C > 0$ such that for each $m \in \bbn$ we have
\begin{align*}
\mathbb{E} \bigg[ \sup_{t \in [0,T]} \| \xi_m(t) \|_{\cald(A)}^{2q} \bigg] \leq C \delta_m^{-(q+1)}.
\end{align*}
\end{corollary}

\begin{proof}
By Proposition \ref{prop-ex-sol-WZ}, for each $m \in \bbn$ the process $\xi_m$ is a solution to the $\cald(A)$-valued integral equation
\begin{align*}
\xi_m(t) = S_t x_0 + \big( \cald(A) \, \text{--} \big) \int_0^t S_{t-s} \bigg( b(\xi_m(s)) + \sum_{j=1}^r \sigma_j(\xi_m(s)) \dot{B}_m^j(s) \bigg) ds, \quad t \in [0,T].
\end{align*}
Therefore, taking into account Assumption \ref{ass-D}, the stated estimate is an immediate consequence of Lemma \ref{lemma-E-B-dot}.
\end{proof}

The following result contributes to \cite[Lemma 2.11]{Nakayama-Support}, where it was shown that merely under Assumption \ref{ass-H} (that is, without imposing Assumption \ref{ass-D}) for each each $\alpha \in (\frac{1}{2p},\frac{1}{2})$ there is a constant $C > 0$ such that for each $l \in \{ 1,\ldots,r \}$ and each $m \in \bbn$ we have
\begin{align*}
\mathbb{E} \bigg[ \sup_{t \in [0,T]} \| \zeta_{m,3}^{j,l}(t) \|^{2p} \bigg] \leq C \delta_m^{p(1-2\alpha)}.
\end{align*}

\begin{proposition}\label{prop-WZ-3}
For each $\kappa \in (0,p)$ there is a constant $C > 0$ such that for each $l \in \{ 1,\ldots,r \}$ and each $m \in \bbn$ we have
\begin{align*}
\mathbb{E} \bigg[ \sup_{t \in [0,T]} \| \zeta_{m,3}^{j,l}(t) \|^{2p} \bigg] \leq C \delta_m^{\kappa}.
\end{align*}
\end{proposition}

\begin{proof}
Let
\begin{align*}
I_m(t):=\delta_m^{-2}\int_0^{[t]_m^-}S_{t-s}
D\sigma_j(\bar{\xi}_m(s))\int_{[s]_m^-}^sS_{s-u}
\sigma_l(\bar{\xi}_m(u))du
K_m(s)ds ,
\end{align*}
where
\begin{align*}
K_m(s):=
\begin{cases}
(B^j([s]_m^+)-B^j([s]_m^-))^2-\delta_m,\ & \text{if $j = l$,}\\
(B^j([s]_m^+)-B^j([s]_m^-))(B^l([s]_m^+)-B^l([s]_m^-)),\ & \text{if $j\neq l$.}
\end{cases}
\end{align*}
Then we have $\zeta_{m,3}^{j,l}(t) = I_m(t)$, and hence, it has to be shown that
\begin{align*}
\mathbb{E} \bigg[ \sup_{t \in [0,T]} \| I_m(t) \|^{2p} \bigg] \leq C \delta_m^{\kappa}.
\end{align*}
Note that
\begin{align*}
I_m(t) = \delta_m^{-2}(I_m^1(t)+I_m^2(t)+I_m^3(t)),
\end{align*}
where the quantity $I_m^1(t)$ is defined as
\begin{align*}
I_m^1(t) &:= \int_0^{[t]_m^-}S_{t-s}
D\sigma_j(\bar{\xi}_m(s))
\\ &\qquad \qquad \int_{[s]_m^-}^sS_{s-u}
(\sigma_l(\bar{\xi}_m(u))-\sigma_l(\xi_m([u]_m^-)))du
K_m(s)ds,
\end{align*}
the quantity $I_m^2(t)$ is defined as
\begin{align*}
I_m^2(t) &:= \int_0^{[t]_m^-}S_{t-s}
(D\sigma_j(\bar{\xi}_m(s))-D\sigma_j(\xi_m([s]_m^-)))
\\ &\qquad \qquad \int_{[s]_m^-}^sS_{s-u}
\sigma_l(\xi_m([u]_m^-))du
K_m(s)ds,
\end{align*}
and the quantity $I_m^3(t)$ is defined as
\begin{align*}
I_m^3(t) &:= \int_0^{[t]_m^-}S_{t-s}
D\sigma_j(\xi_m([s]_m^-))
\\ &\qquad \qquad \int_{[s]_m^-}^sS_{s-u}
\sigma_l(\xi_m([u]_m^-))du
K_m(s)ds.
\end{align*}
Next, we define
\begin{align*}
\tilde{K}_m(k):=
\begin{cases}
(B^j((k+1)\delta_m)-B^j(k\delta_m))^2-\delta_m,\ & \text{if $j = l$,}\\
(B^j((k+1)\delta_m)-B^j(k\delta_m))(B^l((k+1)\delta_m)-B^l(k\delta_m)),\ & \text{if $j\neq l$,}
\end{cases}
\end{align*}
and we choose constants $\pi,\theta \in (1,\infty)$ such that
\begin{align*}
p - \frac{1}{\pi} > \kappa \quad \text{and} \quad \frac{1}{\pi} + \frac{1}{\theta} = 1.
\end{align*}
Noting that $B^j$ and $B^l$ are independent for $j \neq l$, by Corollary \ref{cor-chi-square} (applied with $q = 2 p \theta$ in case $j = l$, and applied twice with $q = p \theta$ in case $j \neq l$) we have
\begin{equation*}
\bbe \left[\vert\tilde{K}_m(k)\vert^{2 p \theta}\right]^{1 / \theta} \leq C\delta_m^{2p} \quad \text{for all $k = 0,\ldots,m-1$.}
\end{equation*}
Now, we set
\begin{equation*}
U_m(t) := \sup_{0\leq u\leq\delta_m}
\Vert (S_u-{\rm Id})\xi_m(t)\Vert.
\end{equation*}
By Lemma \ref{lemma-est-semigroup}, for each $t \in [0,T]$ we have
\begin{align*}
U_m(t)^{2p\pi} &= \sup_{0\leq u\leq\delta_m}
\Vert (S_u-{\rm Id})\xi_m(t)\Vert^{2p\pi} \leq C \sup_{0\leq u\leq\delta_m}
u^{2p\pi} \Vert\xi_m(t)\Vert_{\cald(A)}^{2p\pi}
\\ &\leq C \delta_m^{2p\pi} \Vert\xi_m(t)\Vert_{\cald(A)}^{2p\pi}.
\end{align*}
Therefore, by Corollary \ref{cor-E-B-dot} (applied with $q = p \pi$) we obtain
\begin{align*}
\bbe\left[ \max_{0\leq k\leq m} U_m(k\delta_m)^{2p\pi}\right] &\leq \bbe \bigg[ \sup_{t \in [0,T]} U_m(t)^{2p\pi} \bigg] \leq C \delta_m^{2p\pi} \, \bbe \bigg[ \sup_{t \in [0,T]} \Vert\xi_m(t)\Vert_{\cald(A)}^{2p\pi} \bigg]
\\ &\leq C \delta_m^{2p\pi} \delta_m^{-(p \pi+1)} = C \delta_m^{p\pi-1}.
\end{align*}
Now we have
\begin{align*}
\Vert I_m^1(t)\Vert
&\leq C\int_0^{T}\int_{[s]_m^-}^s
\Vert (S_{u-[s]_m^-}-{\rm Id})\xi_m([s]_m^-)\Vert du
\vert K_m(s)\vert ds\\
&\leq C\int_0^{T}(s-[s]_m^-)
U_m([s]_m^-)
\vert K_m(s)\vert ds\\
&=
C\frac{\delta_m^2}2\sum_{k=0}^{m-1}
U_m(k\delta_m)
\vert
\tilde{K}_m(k)\vert.
\end{align*}
Hence we obtain
\begin{align*}
\Vert I_m^1(t)\Vert^{2p}
&\leq
C\left(\frac{\delta_m^2}2\right)^{2p}m^{2p-1}
\times\sum_{k=0}^{m-1}
U_m(k\delta_m)^{2p}
\vert
\tilde{K}_m(k)\vert^{2p}
\\ &\leq C\left(\frac{\delta_m^2}2\right)^{2p}m^{2p} \max_{0 \leq k \leq m}
U_m(k\delta_m)^{2p}
\vert
\tilde{K}_m(k)\vert^{2p}.
\end{align*}
By H\"{o}lder's inequality we conclude that
\begin{align*}
&\bbe \bigg[ \sup_{t \in [0,T]}\Vert I_m^1(t)\Vert^{2p} \bigg]
\\ &\leq
C\left(\frac{\delta_m^2}2\right)^{2p}m^{2p} \bbe\left[ \max_{0\leq k\leq m} U_m(k\delta_m)^{2p\pi}\right]^{1/\pi} \bbe \left[\vert\tilde{K}_m(k)\vert^{2p\theta}\right]^{1/\theta}
\\ &\leq C \delta_m^{4p} \delta_m^{-2p} \delta_m^{p - \frac{1}{\pi}} \delta_m^{2p} \leq C \delta_m^{4p} \delta_m^{\kappa}.
\end{align*}
Now we have
\begin{align*}
\Vert I_m^2(t)\Vert
&\leq C \int_0^T
\Vert(S_{s-[s]_m^-}-{\rm Id})\xi_m([s]_m^-)\Vert
(s-[s]_m^-)\vert K_m(s)\vert ds\\
&\leq
C \int_0^T
U_m([s]_m^-)
(s-[s]_m^-)\vert K_m(s)\vert ds.
\end{align*}
So by the same argument as above, we also have
\begin{equation*}
\bbe \bigg[ \sup_{t \in [0,T]}\Vert I_m^2(t)\Vert^{2p} \bigg]
\leq
C \delta_m^{4p} \delta_m^{\kappa}.
\end{equation*}
Next we have
\begin{equation*}
I_m^3(t)
=
\sum_{k=0}^{\delta_m^{-1}[t]_m^--1}
\int_{k\delta_m}^{(k+1)\delta_m}
S_{t-s}D\sigma_j(\xi_m(k\delta_m))\int_{k\delta_m}^sS_{s-u}
\sigma_l(\xi_m(k\delta_m))du
K_m(s)ds.
\end{equation*}
Let us consider the process $M^m = (M_n^m)_{n=0,1,\ldots,m}$ defined by
$$
M_n^m:=
\sum_{k=0}^{n-1}
\int_{k\delta_m}^{(k+1)\delta_m}
S_{t-s}D\sigma_j(\xi_m(k\delta_m))\int_{k\delta_m}^sS_{s-u}
\sigma_l(\xi_m(k\delta_m))du
K_m(s)ds
$$
for $n=1,2,\ldots,m$ with $M_0^m=0$. Then $M^m$ is a $(\calf_{n \delta_m})$-martingale, because $\bbe[K_m(s)] = 0$ for all $s \in [0,T]$, and $\xi_m(k \delta_m)$ is $\calf_{k \delta_m}$-measurable and $\tilde{K}_m(k)$ is independent of $\calf_{k \delta_m}$ for all $k = 0,\ldots,m-1$. Furthermore, by the Lemma 2.10 in the paper \cite{Nakayama-Support}, we have
\begin{align*}
&\bbe\left[\sup_{t \in [0,T]}\Vert I_m^3(t)\Vert^{2p}\right]
=
\bbe\left[\max_{1\leq n\leq m}\Vert M_n^m(t)\Vert^{2p}\right]
\leq
C_p \bbe \left[\left(\sum_{k=0}^{m-1}\Vert M_{k+1}^m-M_k^m\Vert^2\right)^p\right]\\
&=
C_p\bbe\Bigg[\Bigg(\sum_{k=0}^{m-1}\bigg\Vert
\int_{k\delta_m}^{(k+1)\delta_m}
S_{t-s}D\sigma_j(\xi_m(k\delta_m))
\\ &\qquad \qquad \qquad \int_{k\delta_m}^sS_{s-u}
\sigma_l(\xi_m(k\delta_m))du
K_m(s)ds
\bigg\Vert^2\Bigg)^p\Bigg],
\end{align*}
where $C_p$ is depending only on $p$, and hence
\begin{align*}
&\bbe\left[\sup_{t \in [0,T]}\Vert I_m^3(t)\Vert^{2p}\right] \leq
C\bbe\left[\left(\sum_{k=0}^{m-1}
\left(\int_{k\delta_m}^{(k+1)\delta_m}
(s-k\delta_m)
\vert K_m(s)\vert ds
\right)^2\right)^p\right]\\
&=
C\left(\frac{\delta_m^2}2\right)^{2p}\bbe\left[\left(\sum_{k=0}^{m-1}
\vert\tilde{K}_m(k)\vert^2\right)^p\right]
\leq
C\left(\frac{\delta_m^2}2\right)^{2p}m^{p-1}
\sum_{k=0}^{m-1}
\bbe\left[\vert\tilde{K}_m(k)\vert^{2p}\right]\\
&\leq
C\left(\frac{\delta_m^2}2\right)^{2p}m^p
\delta_m^{2p}
\leq C\delta_m^{5p}.
\end{align*}
Combining above results, we have
\begin{equation*}
\bbe\left[\sup_{t \in [0,T]}\Vert I_m(t)\Vert^{2p}\right]
\leq C \delta_m^{-4p} ( \delta_m^{4p} \delta_m^{\kappa} + \delta_m^{5p} ) = C ( \delta_m^{\kappa} + \delta_m^p ) \leq C \delta_m^{\kappa},
\end{equation*}
completing the proof.
\end{proof}

\begin{proposition}\label{prop-WZ-4}
There is a constant $C > 0$ such that for each $m \in \bbn$ and each $v \in [0,T]$ we have
\begin{align*}
\mathbb{E} \bigg[ \sup_{t \in [0,v]} \| \zeta_{m,4}^j(t) \|^{2p} \bigg] \leq C \bigg( \int_0^v \bbe \bigg[ \sup_{t \in [0,u]} \| \xi_m(t) - Y_m(t) \|^{2p} \bigg] du + \delta_m^{2p} \bigg).
\end{align*}
\end{proposition}

\begin{proof}
The proof is similar to that of \cite[Lemma 2.12]{Nakayama-Support}. Note that
\begin{align*}
\zeta_{m,4}^j(t) = I_{m,1}^j(t) + I_{m,2}^j(t) + I_{m,3}^j(t) + I_{m,4}^j(t),
\end{align*}
where the quantity $I_{m,1}^j(t)$ is given by
\begin{align*}
I_{m,1}^j(t) &= \frac{1}{\delta_m} \int_0^{[t]_m^-} S_{t-s} D\sigma_j(\bar{\xi}_m(s)) \bigg( \int_{[s]_m^-}^s S_{s-u} \sigma_j(\bar{\xi}_m(u))du \bigg) ds
\\ &\quad - \frac{1}{\delta_m} \int_0^{[t]_m^-} S_{t-s} D\sigma_j(\bar{Y}_m(s)) \bigg( \int_{[s]_m^-}^s S_{s-u} \sigma_j(\bar{Y}_m(u))du \bigg) ds,
\end{align*}
the quantity $I_{m,2}^j(t)$ is given by
\begin{align*}
I_{m,2}^j(t) &= \frac{1}{\delta_m} \int_0^{[t]_m^-} S_{t-s} D\sigma_j(\bar{Y}_m(s)) \bigg( \int_{[s]_m^-}^s S_{s-u} \sigma_j(\bar{Y}_m(u))du \bigg) ds
\\ &\quad - \frac{1}{\delta_m} \int_0^{[t]_m^-} S_{t-s} D\sigma_j(Y_m([s]_m^-)) \bigg( \int_{[s]_m^-}^s S_{s-u} \sigma_j(Y_m([u]_m^-))du \bigg) ds,
\end{align*}
the quantity $I_{m,3}^j(t)$ is given by
\begin{align*}
I_{m,3}^j(t) &= \frac{1}{\delta_m} \int_0^{[t]_m^-} S_{t-s} D\sigma_j(Y_m([s]_m^-)) \bigg( \int_{[s]_m^-}^s S_{s-u} \sigma_j(Y_m([u]_m^-))du \bigg) ds
\\ &\quad - \frac{1}{\delta_m} \int_0^{[t]_m^-} (s - [s]_m^-) S_{t - [s]_m^-} D \sigma_j(Y_m([s]_m^-)) \sigma_j(Y_m([s]_m^-)) ds,
\end{align*}
and the quantity $I_{m,4}^j(t)$ is given by
\begin{align*}
I_{m,4}^j(t) &= \frac{1}{\delta_m} \int_0^{[t]_m^-} (s - [s]_m^-) S_{t - [s]_m^-} D \sigma_j(Y_m([s]_m^-)) \sigma_j(Y_m([s]_m^-)) ds
\\ &\quad - \frac{1}{2} \int_0^{[t]_m^-} S_{t-s} D \sigma_j(Y_m([s]_m^-)) \sigma_j(Y_m([s]_m^-)) ds.
\end{align*}
We have
\begin{align*}
\| I_{m,1}^j(t) \| &\leq \frac{1}{\delta_m} \bigg\| \int_0^{[t]_m^-} S_{t-s} D\sigma_j(\bar{\xi}_m(s))
\\ &\qquad\qquad \bigg( \int_{[s]_m^-}^s S_{s-u} \big( \sigma_j(\bar{\xi}_m(u)) - \sigma_j(\bar{Y}_m(u)) \big) du \bigg) ds \bigg\|
\\ &\quad + \frac{1}{\delta_m} \bigg\| \int_0^{[t]_m^-} S_{t-s} \big( D \sigma_j(\bar{\xi}_m(s)) - D\sigma_j(\bar{Y}_m(s)) \big)
\\ &\qquad\qquad \bigg( \int_{[s]_m^-}^s S_{s-u} \sigma_j(\bar{Y}_m(u))du \bigg) ds \bigg\|,
\end{align*}
and hence
\begin{align*}
\| I_{m,1}^j(t) \| &\leq C \bigg[ \frac{1}{\delta_m} \int_0^{[t]_m^-} \bigg( \int_{[s]_m^-}^s \| \bar{\xi}_m(u) - \bar{Y}_m(u) \| du \bigg) ds
\\ &\qquad + \int_0^{[t]_m^-} \| \bar{\xi}_m(s) - \bar{Y}_m(s) \| ds \bigg]
\\ &\leq C \int_0^{[v]_m^-} \| \xi_m([s]_m^-) - Y_m([s]_m^-) \| ds \leq C \int_0^{[v]_m^-} \sup_{t \in [0,u]} \| \xi_m(t) - Y_m(t) \| du.
\end{align*}
for all $0 \leq t \leq v \leq T$. Therefore, we get
\begin{align*}
\bbe \bigg[ \sup_{t \in [0,v]} \| I_{m,1}^j(t) \|^{2p} \bigg] \leq C \int_0^v \bbe \bigg[ \sup_{t \in [0,u]} \| \xi_m(t) - Y_m(t) \|^{2p} \bigg] du.
\end{align*}
Furthermore, we have
\begin{align*}
\| I_{m,2}^j(t) \| &= \frac{1}{\delta_m} \bigg\| \int_0^{[t]_m^-} S_{t-s} D\sigma_j(\bar{Y}_m(s))
\\ &\qquad\qquad \bigg( \int_{[s]_m^-}^s S_{s-u} \big( \sigma_j(\bar{Y}_m(u)) - \sigma_j(Y_m([u]_m^-)) \big) du \bigg) ds \bigg\|
\\ &\quad + \frac{1}{\delta_m} \bigg\| \int_0^{[t]_m^-} S_{t-s} \big( D\sigma_j(\bar{Y}_m(s)) - D\sigma_j(Y_m([s]_m^-)) \big)
\\ &\qquad\qquad \bigg( \int_{[s]_m^-}^s S_{s-u} \sigma_j(Y_m([u]_m^-))du \bigg) ds \bigg\|,
\end{align*}
and hence
\begin{align*}
\| I_{m,2}^j(t) \| &\leq C \bigg[ \frac{1}{\delta_m} \int_0^{[t]_m^-} \bigg( \int_{[s]_m^-}^s \| (S_{u - [u]_m^-} - {\rm Id}) Y_m([u]_m^-) \| du \bigg) ds
\\ &\quad + \int_0^{[t]_m^-} \| (S_{s - [s]_m^-} - {\rm Id}) Y_m([s]_m^-) \| ds \bigg]
\\ &\leq C \int_0^{[t]_m^-} \sup_{[s]_m^- \leq u \leq s} \| (S_{u - [u]_m^-} - {\rm Id}) Y_m([u]_m^-) \| ds. 
\end{align*}
Therefore, by Lemma \ref{lemma-est-semigroup} we obtain
\begin{align*}
&\bbe \bigg[ \sup_{t \in [0,T]} \| I_{m,2}^j(t) \|^{2p} \bigg] \leq C \bbe \bigg[ \sup_{t \in [0,T]} \| (S_{t - [t]_m^-} - {\rm Id}) Y_m([t]_m^-) \|^{2p} \bigg]
\\ &\leq C \delta_m^{2p} \bbe \bigg[ \sup_{t \in [0,T]} \| Y_m([t]_m^-) \|_{\cald(A)}^{2p} \bigg].
\end{align*}
Therefore, by virtue of Lemma \ref{lemma-convolution} -- applied with the separable Hilbert space $(\cald(A),\| \cdot \|_{\cald(A)})$ -- and Assumption \ref{ass-D} we obtain
\begin{align*}
\bbe \bigg[ \sup_{t \in [0,T]} \| I_{m,2}^j(t) \|^{2p} \bigg] \leq C \delta_m^{2p}.
\end{align*}
Furthermore, we have
\begin{align*}
\| I_{m,3}^j(t) \| &\leq \frac{1}{\delta_m} \bigg\| \int_0^{[t]_m^-} S_{t-s} D \sigma_j(Y_m([s]_m^-))
\\ &\qquad\qquad \bigg( \int_{[s]_m^-}^s (S_{s-u} - {\rm Id}) \sigma_j(Y_m([s]_m^-)) du \bigg) ds \bigg\|
\\ &\quad + \frac{1}{\delta_m} \bigg\| \int_0^{[t]_m^-} (s - [s]_m^-) \big( S_{t-s} - S_{t - [s]_m^-} \big)
\\ &\qquad\qquad D\sigma_j(Y_m([s]_m^-)) \sigma_j(Y_m([s]_m^-)) ds \bigg\|,
\end{align*}
and hence
\begin{align*}
\| I_{m,3}^j(t) \| &\leq C \bigg[ \frac{1}{\delta_m} \int_0^{[t]_m^-} \bigg( \int_{[s]_m^-}^s \| ( S_{s-u} - {\rm Id} ) \sigma_j(Y_m([s]_m^-)) \| du \bigg) ds
\\ &\qquad + \int_0^{[t]_m^-} \| ( S_{t-s} - S_{t - [s]_m^-} ) D\sigma_j(Y_m([s]_m^-)) \sigma_j(Y_m([s]_m^-)) \| ds \bigg].
\end{align*}
Therefore, by Lemma \ref{lemma-est-semigroup} and Assumption \ref{ass-D}, we obtain
\begin{align*}
\bbe \bigg[ \sup_{t \in [0,T]} \| I_{m,3}^j(t) \|^{2p} \bigg] \leq C \delta_m^{2p}.
\end{align*}
Moreover, we have
\begin{align*}
&\frac{1}{\delta_m} \int_0^{[t]_m^-} (s - [s]_m^-) S_{t - [s]_m^-} D \sigma_j(Y_m([s]_m^-)) \sigma_j(Y_m([s]_m^-)) ds
\\ &= \frac{1}{\delta_m} \sum_{k=0}^{\delta_m^{-1} [t]_m^- - 1} \int_{k \delta_m}^{(k+1)\delta_m} (s - k \delta_m) S_{t - k \delta_m} D \sigma_j(Y_m(k \delta_m)) \sigma_j(Y_m(k \delta_m)) ds
\\ &= \frac{\delta_m}{2} \sum_{k=0}^{\delta_m^{-1} [t]_m^- - 1} S_{t - k \delta_m} D \sigma_j(Y_m(k \delta_m)) \sigma_j(Y_m(k \delta_m))
\\ &= \frac{1}{2} \sum_{k=0}^{\delta_m^{-1} [t]_m^- - 1} \int_{k \delta_m}^{(k+1)\delta_m} S_{t - k \delta_m} D \sigma_j(Y_m(k \delta_m)) \sigma_j(Y_m(k \delta_m)) ds
\\ &= \frac{1}{2} \int_0^{[t]_m^-} S_{t - [s]_m^-} D \sigma_j(Y_m([s]_m^-)) \sigma_j(Y_m([s]_m^-)) ds.
\end{align*}
Therefore, we have
\begin{align*}
I_{m,4}^j(t) = \frac{1}{2} \int_0^{[t]_m^-} \big( S_{t - [s]_m^-} - S_{t-s} \big) D \sigma_j(Y_m([s]_m^-)) \sigma_j(Y_m([s]_m^-)) ds.
\end{align*}
Consequently, by Lemma \ref{lemma-est-semigroup} and Assumption \ref{ass-D}, we obtain
\begin{align*}
\bbe \bigg[ \sup_{t \in [0,T]} \| I_{m,4}^j(t) \|^{2p} \bigg] \leq C \delta_m^{2p},
\end{align*}
completing the proof.
\end{proof}

\begin{proposition}\label{prop-WZ-5}
There is a constant $C > 0$ such that for each $m \in \bbn$ and each
$v \in [0,T]$ we have
\begin{align*}
\mathbb{E} \bigg[ \sup_{t \in [0,v]} \| \zeta_{m,5}(t) \|^{2p} \bigg] \leq C \bigg( \int_0^v \bbe \bigg[ \sup_{t \in [0,u]} \| \xi_m(t) - Y_m(t) \|^{2p} \bigg] du + \delta_m^{p-1} \bigg).
\end{align*}
\end{proposition}

\begin{proof}
Since $b$ is Lipschitz continuous, we obtain
\begin{align*}
&\mathbb{E} \bigg[ \sup_{t \in [0,v]} \| \zeta_{m,5}(t) \|^{2p} \bigg] = \mathbb{E} \Bigg[ \sup_{t \in [0,v]} \bigg\| \int_0^{[t]_m^-} S_{t-s} \big( b(\xi_m(s)) - b(Y_m([s]_m^-)) \big) ds \bigg\|^{2p} \Bigg]
\\ &\leq C \int_0^v \bbe[ \| \xi_m(s) - Y_m([s]_m^-) \|^{2p} ] ds \leq C \big( I_1(v) + I_2(v) + I_3(v) \big),
\end{align*}
where
\begin{align*}
I_1(v) &= \int_0^v \bbe \bigg[ \sup_{t \in [0,u]} \| \xi_m(t) - Y_m(t) \|^{2p} \bigg] du,
\\ I_2(v) &= \int_0^v \bbe \bigg[ \sup_{t \in [0,u]} \| Y_m(t) - \bar{Y}_m(t) \|^{2p} \bigg] du,
\\ I_3(v) &= \int_0^v \bbe \bigg[ \sup_{t \in [0,u]} \| \bar{Y}_m(t) - Y_m([t]_m^-) \|^{2p} \bigg] du.
\end{align*}
Therefore, applying Proposition \ref{prop-Euler-1} and Lemma \ref{lemma-Euler-event-1} completes the proof.
\end{proof}

Now, the proof of Theorem \ref{thm-Euler-WZ} is an immediate consequence of the decompositions (\ref{decomp-Euler-WZ}), (\ref{xi-part-3}), Propositions \ref{prop-Euler-1}, \ref{prop-Euler-WZ-pre}, Propositions \ref{prop-WZ-1}--\ref{prop-WZ-5} and Gronwall's inequality.

\section{An example: The HJMM equation}\label{sec-HJMM}

As an example of our main result, let us consider the HJMM (Heath-Jarrow-Morton-Musiela) equation from mathematical finance. This is a SPDE which models the term structure of interest rates in a market of zero coupon bonds.

Let us briefly introduce the model. A zero coupon bond with maturity $T$ is a financial asset that pays the holder one monetary unit at $T$. Its price at $t \leq T$ can be written as the continuous discounting of one unit of the domestic currency
\begin{align*}
P(t,T) = \exp \bigg( -\int_t^T f(t,s) ds \bigg),
\end{align*}
where $f(t,T)$ is the rate prevailing at time $t$ for instantaneous borrowing at time $T$, also called the forward rate for date $T$.

After transforming the original HJM (Heath-Jarrow-Morton) dynamics of the forward rates (see \cite{HJM}) by means of the Musiela parametrization $r_t(x) = f(t,t+x)$ (see \cite{Musiela}), the instantaneous forward rate $r_t(x)$ with maturity time $x$ from observing time $t$ can be considered as a mild solution to the HJMM (Heath-Jarrow-Morton-Musiela) equation
\begin{align}\label{HJMM}
\left\{
\begin{array}{rcl}
dr(t) & = & \big( \frac{\partial}{\partial x}r(t) +
\alpha(r(t),v(t)) \big) dt
+\sum_{j=1}^r\gamma_j(r(t),v(t))dB^j(t) \medskip
\\ v(t) & = & \mu(v(t))dt+\sum_{j=1}^r\lambda_j(v(t))dB^j(t) \medskip
\\ r(0) & = & r_0 \medskip
\\ v(0) & = & v_0.
\end{array}
\right.
\end{align}
Note that we consider the HJMM equation (\ref{HJMM}) with stochastic volatility. More precisely, the functions $(\gamma_j)_{j=1,\ldots,r}$ are called volatility functions; they represent the degree of variation of the instantaneous forward rate. If $v(t)=t$, $t \in \bbr_+$, which corresponds to $\mu\equiv1$ and $\lambda_j\equiv0$ for all $j=1,\ldots,r$, then the model is called a local volatility model. In this case,
the volatility of the instantaneous forward rate with each maturity depends on the observation time and the interest rate curve itself.
If $v(t)$ is not deterministic, then the model is called a stochastic volatility model, which fits more with the real interest rate market.

In order to ensure absence of arbitrage in the bond market, we consider the HJMM equation (\ref{HJMM}) under a martingale measure. Then the drift term is given by the so-called HJM drift condition
\begin{align}\label{HJM-drift}
\alpha(h,v) = \sum_{j=1}^r \gamma_j(h,v) \int_0^{\bullet} \gamma_j(h,v)(\eta) d\eta.
\end{align}
We refer, e.g., to \cite{fillnm} for further details concerning the derivation of the HJMM equation (\ref{HJMM}) and the HJM drift condition (\ref{HJM-drift}).

The precise mathematical formulation of our model is as follows. We fix an arbitrary constant $\beta > 0$. Let $\tilde{H}_{\beta}$ be the space of all absolutely continuous functions $h : \mathbb{R}_+ \rightarrow \mathbb{R}$ such that
\begin{align}\label{def-norm}
\| h \|_{\beta} := \bigg( |h(0)|^2 + \int_{\mathbb{R}_+} |h'(x)|^2
e^{\beta x} dx \bigg)^{\frac{1}{2}} < \infty.
\end{align}
This kind of space was introduced in \cite[Sec. 5.1]{fillnm}, where the following properties have been proven:
\begin{itemize}
\item The space $(\tilde{H}_{\beta},\| \cdot \|_{\beta})$ is a separable
Hilbert space.

\item For each $x \in \mathbb{R}_+$ the point
evaluation $h \mapsto h(x) : \tilde{H}_{\beta} \rightarrow \mathbb{R}$ is a
continuous linear functional.

\item The semigroup $(\tilde{S}_t)_{t\geq0}$ of right shifts in $\tilde{H}_{\beta}$ defined by
$\tilde{S}_th(x)=h(x+t)$ is a $C_0$-semigroup
with infinitesimal generator $\tilde{A}$ given by
$\tilde{A}h(x)=h'(x)$ on the domain
${\cald}_{\beta}(\tilde{A})=\{h\in \tilde{H}_{\beta} : h'\in\tilde{H}_{\beta}\}$.

\item For each $h \in \tilde{H}_{\beta}$ the limit $h(\infty) := \lim_{x \rightarrow \infty} h(x)$ exists, and
\begin{align*}
\tilde{H}_{\beta}^0 := \{ h \in \tilde{H}_{\beta} : h(\infty) = 0 \}
\end{align*}
is a closed subspace of $\tilde{H}_{\beta}$.
\end{itemize}
Let us fix an additional index $\beta' > \beta$. We have the following additional result.

\begin{lemma}\label{lemma-mult}
The following statements are true:
\begin{enumerate}
\item The multiplication operator $m : \tilde{H}_{\beta} \times \tilde{H}_{\beta} \to \tilde{H}_{\beta}$ given by $m(h,g) = hg$ is a continuous bilinear operator.

\item We have $m(\cald_{\beta}(\tilde{A}) \times \cald_{\beta}(\tilde{A})) \subset \cald_{\beta}(\tilde{A})$, and the restriction of $m$ is a continuous bilinear operator with respect to the graph norm.

\item We have $\tilde{H}_{\beta'} \subset \tilde{H}_{\beta}$ with continuous embedding.

\item The integral operator $\mathcal{I} : \tilde{H}_{\beta'}^0 \to \tilde{H}_{\beta}$ given by $\mathcal{I}
h := \int_0^{\bullet} h(\eta) d\eta$ is a continuous linear operator.

\item We have $\mathcal{I}(\cald_{\beta'}(\tilde{A}) \cap \tilde{H}_{\beta'}^0) \subset \cald_{\beta}(\tilde{A})$, and the restriction is a continuous linear operator with respect to the corresponding graph norms.
\end{enumerate}
\end{lemma}

\begin{proof}
This is a consequence of \cite[Thm. 4.1 and Lemmas 4.2, 4.3]{Tappe-Wiener}.
\end{proof}

Now, let us assume the following.

\begin{assumption}\label{HJM-assumption}
We suppose that the following conditions are satisfied:
\begin{itemize}
\item We have $\gamma_j \in C_b^2(\tilde{H}_{\beta} \times \bbr; \tilde{H}_{\beta'}^0)$ for each $j=1,\ldots,r$.

\item We have $\gamma_j(\cald_{\beta}(\tilde{A}) \times \bbr) \subset \cald_{\beta'}(\tilde{A})$ for each $j=1,\ldots,r$.

\item We have $\gamma_j|_{\cald_{\beta}(\tilde{A}) \times \bbr} \in C_b^2(\cald_{\beta}(\tilde{A}) \times \bbr; \cald_{\beta'}(\tilde{A}))$ for each $j=1,\ldots,r$ with respect to the corresponding graph norms.

\item $\mu : \bbr \to \bbr$ is Lipschitz continuous and bounded.

\item We have $\lambda_j \in C_b^2(\bbr)$ for each $j=1,\ldots,r$.
\end{itemize}
\end{assumption}

Now we can consider the HJM model in our SPDE framework, and rewrite the HJMM equation (\ref{HJMM}) as
\begin{align}\label{HJMM-2}
\left\{
\begin{array}{rcl}
dX(t) & = & \big( AX(t)+b(X(t)) \big) dt+\sum_{j=1}^r\sigma_j(X(t))dB^j(t) \medskip
\\ X(0) & = & x_0
\end{array}
\right.
\end{align}
on the separable Hilbert space $H=\tilde{H}_{\beta}\times{\mathbb{R}}$ with the notation
\begin{align*}
&x_0:=\begin{pmatrix}r_0\\v_0\end{pmatrix}\in \cald(A) := {\cald}_{\beta}(\tilde{A})\times\mathbb{R},\\
&X(t):=\begin{pmatrix}r(t)\\v(t)\end{pmatrix}\in H,\\
&S_tx:=\begin{pmatrix}\tilde{S}_th\\v\end{pmatrix}
\text{ for }x=\begin{pmatrix}h\\v\end{pmatrix}\in H,\\
&Ax:=\begin{pmatrix}\tilde{A}h\\0\end{pmatrix}
\text{ for }x=\begin{pmatrix}h\\v\end{pmatrix}\in \cald(A) := {\cald}_{\beta}(\tilde{A})\times\mathbb{R},\\
&b:=\begin{pmatrix}b^1\\ b^2\end{pmatrix}\colon H\to H,
\quad b^1(h,v)=\alpha(h,v),\quad b^2(h,v)=\mu(v),\\
&\sigma_j
:=\begin{pmatrix}\sigma_j^1\\ \sigma_j^2\end{pmatrix}\colon H\to H,
\quad\sigma_j^1(h,v)=\gamma_j(h,v),\quad\sigma_j^2(h,v)=\lambda_j(v).
\end{align*}
Note that the HJM drift term (\ref{HJM-drift}) has the representation
\begin{align*}
\alpha = \sum_{j=1}^r m ( \gamma_j, \mathcal{I} \gamma_j ).
\end{align*}
Taking into account the Leibniz rule (see, for example \cite[Thm. 2.4.4]{Abraham}), by virtue of Lemma \ref{lemma-mult} and Assumption \ref{HJM-assumption} we obtain that Assumptions \ref{ass-H} and \ref{ass-D} are fulfilled. Consequently, Theorem \ref{thm-main} applies to the HJMM equation (\ref{HJMM-2}) and provides the convergence rate for the corresponding Wong-Zakai approximations. By Proposition~\ref{prop-ex-sol-WZ} and identity (\ref{derivative}) the Wong-Zakai approximations $\xi_m = (\varrho_m,\zeta_m)$ for $m \in \bbn$ are the solutions to the integral equations
\begin{align*}
\varrho_m(t,x) &= \varrho_m(k \delta_m, x+t-k\delta_m) + \int_{k \delta_m}^t \bigg( \alpha(\varrho_m(t),\zeta_m(s),x+t-s)
\\ &\quad - \beta(\varrho_m(s),\zeta_m(s),x+t-s) 
\\ &\quad + \sum_{j=1}^r \frac{B^j((k+1)\delta_m)-B^j(k\delta_m)}{\delta_m} \gamma_j(\varrho_m(s),\zeta_m(s),x+t-s) \bigg) ds,
\\ \zeta_m(t) &= \zeta_m(k \delta_m) + \int_{k \delta_m}^t \bigg( \mu(\zeta_m(s)) - \frac{1}{2} \sum_{j=1}^r \lambda_j(\zeta_m(s)) \lambda_j'(\zeta_m(s))
\\ &\quad + \sum_{j=1}^r \frac{B^j((k+1)\delta_m)-B^j(k\delta_m)}{\delta_m} \lambda_j(\zeta_m(s)) \bigg) ds
\end{align*}
for $t \in [k \delta_m, (k+1) \delta_m]$ and $k=0,\ldots,m-1$. Note that $x \in \bbr_+$ in the second argument of $\varrho_m(t,x)$ denotes the point evaluation of the interest rate curve. Moreover, we have used the notation
\begin{align*}
\beta(r,v) := \frac{1}{2} \sum_{j=1}^r \big( D_r \gamma_j(r,v) \gamma_j(r,v) + D_v \gamma_j(r,v) \lambda_j(v) \big).
\end{align*}

\section{Further examples}\label{sec-natural}

In this section, we treat two further examples arising from natural sciences. Before presenting these examples, let us recall an auxiliary result for the infinitesimal generators of strongly continuous semigroups. As in the previous sections, let $A : \cald(A) \subset H \to H$ be the infinitesimal generator of a $C_0$-semigroup $(S_t)_{t \geq 0}$ on the separable Hilbert space $H$.

\begin{lemma}\label{lemma-Pazy}\cite[Thm. 3.1.1]{Pazy}
Let $R \in L(H)$ be a continuous linear operator, and let the linear operator $B : \cald(B) \subset H \to H$ be given by $\cald(B) := \cald(A)$ and $B := A + R$. Then $B$ is the generator of a $C_0$-semigroup on $H$. 
\end{lemma}

As our first example of this section, we consider the stochastic quantization of the free Euclidean quantum field (cf. \cite[Ex.~1.0.1]{Prevot-Roeckner})
\begin{align}\label{SPDE-quantum}
\left\{
\begin{array}{rcl}
dX(t) & = & (\Delta - m^2) X(t) dt + \sum_{j=1}^r \sigma_j dB^j(t)
\medskip
\\ X(0) & = & x_0,
\end{array}
\right.
\end{align}
where $m \in \bbr_+$ denotes ``mass'', and the volatilities $\sigma_1,\ldots,\sigma_r \in \cald(\Delta)$ are constant. Here we choose the state space $H = L^2(\bbr^d)$, and the Laplace operator $\Delta$ is defined on the domain $\cald(\Delta) = W^2(\bbr^d)$. Taking into account Lemma \ref{lemma-Pazy}, we see that Assumptions \ref{ass-H} and \ref{ass-D} are fulfilled, and hence Theorem \ref{thm-main} applies to the SPDE (\ref{SPDE-quantum}) and provides the convergence rate for the corresponding Wong-Zakai approximations.

Another example is the stochastic cable equation (cf. \cite[Ex.~0.8]{Da_Prato})
\begin{align}\label{SPDE-cable}
\left\{
\begin{array}{rcl}
dV(t) & = & \frac{1}{\tau} (\lambda^2 \Delta V(t) - V(t)) dt + \sum_{j=1}^r \sigma_j dB^j(t)
\medskip
\\ V(0) & = & x_0,
\end{array}
\right.
\end{align}
where $\lambda > 0$ denotes the length constant, $\tau > 0$ denotes the time constant of the electric cable, and the volatilities $\sigma_1,\ldots,\sigma_r \in \cald(\Delta)$ are constant. Here we choose the state space $H = L^2((0,\pi))$, and the Laplace operator $\Delta$ is defined on the domain $\cald(\Delta) = W^2((0,\pi)) \cap H_0^1((0,\pi))$. Taking into account Lemma \ref{lemma-Pazy}, we see that Assumptions \ref{ass-H} and \ref{ass-D} are fulfilled, and hence Theorem \ref{thm-main} applies to the SPDE (\ref{SPDE-cable}) and provides the convergence rate for the corresponding Wong-Zakai approximations.

\section*{Acknowledgement}

We are grateful to Josef Teichmann for initiating this research topic, and for his invaluable assistance and discussions. We also wish to thank Ludwig Baringhaus for his advice regarding the reference \cite{Wilks} used for the calculation of the expectation in Lemma~\ref{lemma-chi-square}.

We are also grateful to an anonymous referee for valuable comments and suggestions.

\nocite{*}

\bibliographystyle{plain}

\bibliography{WZ}

\end{document}